\documentclass[11pt]{article}

\input xy

\xyoption{all}

\usepackage{amssymb,amsbsy,amsthm,amsmath,mathrsfs,graphicx,times}

\newtheorem{thm}{Theorem}[section]
\newtheorem{lem}[thm]{Lemma}
\newtheorem{prop}[thm]{Proposition}

\newtheorem{dfn}[thm]{Definition}

\newtheorem*{thm*}{Theorem}
\theoremstyle{remark}

\newtheorem*{rmks}{Remarks}

\renewcommand{\bf}[1]{\mathbf{#1}}
\renewcommand{\rm}[1]{\mathrm{#1}}

\newcommand{\bbZ}{\mathbb{Z}}

\newcommand{\sfE}{\mathsf{E}}

\renewcommand{\d}{\mathrm{d}}

\renewcommand{\L}{\Lambda}

\renewcommand{\a}{\alpha}
\renewcommand{\b}{\beta}

\renewcommand{\l}{\lambda}
\newcommand{\s}{\sigma}
\renewcommand{\phi}{\varphi}

\newcommand{\fin}{\nolinebreak\hspace{\stretch{1}}$\lhd$}

\renewcommand{\t}[1]{\tilde{#1}}
\renewcommand{\to}{\longrightarrow}

\newcommand{\actson}{\curvearrowright}
\newcommand{\notto}{\,\,\not\!\!\to}

\newcommand{\uhr}{\!\upharpoonright}

\begin{document}

\title{\textbf{A proof of Walsh's convergence theorem using couplings}}

\author{Tim Austin\footnote{Research supported by a fellowship from the Clay Mathematics Institute}\\ \\ \small{Courant Institute, New York University}\\ \small{New York, NY 10012, USA}\\ \small{\texttt{tim@cims.nyu.edu}}}

\date{}

\maketitle




\begin{abstract}
Walsh has recently proved the norm convergence of all nonconventional ergodic averages involving polynomial sequences in discrete nilpotent acting groups. He deduces this convergence from an equivalent, `finitary' assertion of stability over arbitrarily long time-intervals for these averages, which is proved by essentially finitary means. The present paper shows how the induction at the heart of Walsh's proof can also be implemented using more classical notions of ergodic theory: in particular, couplings and characteristic factors.
\end{abstract}

\section{Introduction}

In his recent breakthrough paper~\cite{Walsh12}, Walsh proved the following convergence result for nonconventional ergodic averages.

\begin{thm}\label{thm:A}
Suppose that $G$ is a countable discrete nilpotent group, that $T:G\actson (X,\mu)$ is a probability-preserving action on a standard Borel space, that $p_1,\ldots,p_k:\bbZ \to G$ is a tuple of polynomial sequences, and that $f_1,\ldots,f_k \in L^\infty(\mu)$.  Then the averages
\[\frac{1}{N}\sum_{n=1}^N (f_1\circ T^{p_1(n)})(f_2\circ T^{p_2(n)})\cdots (f_k\circ T^{p_k(n)})\]
converge in $\|\cdot\|_2$ as $N\to\infty$.
\end{thm}

This answers a question first published by Bergelson and Leibman in~\cite{BerLei02}, and before that promoted in person by Furstenberg.  It caps a long sequence of partial results: more complete references are given in Walsh's paper.  For $k\geq 2$, the corresponding question of pointwise convergence has been resolved in a few special cases~\cite{Bou90,Ass10,ChuFra12}, but those methods seem to fall far short of the general case.

Unlike most earlier approaches to such questions, Walsh first converts this problem into a more `finitary' one.  It asks for some quantitative control (at least in principle) over how long one must wait before reaching a very long interval of times $N$ throughout which the averages are very stable.  This conversion bears some resemblance to Tao's earlier proof of an important special case in~\cite{Tao08(nonconv)}, but Walsh does not use a completely finitary approach, as Tao does.  More recently, Zorin-Kranich has shown in~\cite{Zor13} how to extend Walsh's result to mappings defined on a more general amenable group, and, more substantially, to averages taken over arbitrary shifted F\o lner sets in that group.

The present note will show how Theorem~\ref{thm:A} can also be proved using some of the ergodic theoretic machinery from those earlier works.  The heart of the proof is still the method of induction newly introduced by Walsh, but the technical superstructure will appear rather different.  In particular, we avoid the conversion of Theorem~\ref{thm:A} into a finitary assertion about the stability of these averages over long time-intervals, as in~\cite[Theorem 3.2]{Walsh12}.  Instead, we show how Walsh's argument can be mimicked in terms of measurable functions on the initially-given probability-preserving system and some extensiosn of it.  Thus, we will assume without proof the main results about finite-complexity tuples from~\cite{Walsh12}.

I believe that it would be easy to generalize the proof below to Zorin-Kranich's setting in~\cite{Zor13}, if the acting groups are all assumed to be discrete. For non-discrete groups some extra technical considerations might arise.   For the sake of simplicity, we will not pursue that extra generality here.

\section{Background}\label{sec:background}

\subsection{Measure theory and ergodic theory}

We shall work throughout with probability-preserving, homeomorphic actions of countable groups on compact metric spaces.  At the level of measure algebras, any probability-preserving action on a countably-generated probability space may be modelled by such a topological action, so this incurs no loss of generality.  If $G$ is the group, then such an action will be calld a \textbf{$G$-space}.  Standard constructions involving these, such as factors, extensions and joinings, will be taken for granted; Glasner's book~\cite{Gla03} provides a thorough reference.

We will also need the following fairly recent result about this setting, concerning the extension of systems to recover actions of larger groups. If $H \leq G$ is an inclusion of countable groups and $(X,\mu,T)$ is a $G$-space, then we write $T^{\uhr H}$ for the restriction of the action $T$ to the subgroup $H$.

\begin{thm}[{\cite[Theorem 2.1]{Aus--commute}}]\label{thm:recoverG}
Suppose $H \leq G$ is an inclusion of countable groups, that $(X,\mu,T)$ is a $G$-space and that
\[(Y,\nu,S) \stackrel{\b}{\to} (X,\mu,T^{\uhr H})\]
is an extension of $H$-spaces.  Then there is an extension of $G$-spaces $(Z,\theta,R)\stackrel{\pi}{\to} (X,\mu,T)$ which admits a commutative diagram of $H$-spaces
\begin{center}
$\phantom{i}$\xymatrix{
(Z,\theta,R^{\uhr H}) \ar^{\pi}[rr]\ar_\a[dr] && (X,\mu,T^{\uhr H})\\
& (Y,\nu,S)\ar_\b[ur].
}
\end{center}
\qed
\end{thm}

Another slightly less standard notion that we will use is the following. Suppose that $X$ and $Y$ are compact metric spaces and that $\mu\in \Pr(X)$, let $\pi:X\times Y\to X$ be the coordinate-projection, and let
\[Q(\mu,Y) := \{\l \in \Pr(X\times Y):\ \pi_\ast\l = \mu\},\]
the set of probability measures on $X\times Y$ that extend $\mu$.  It always contains $\mu\otimes \delta_y$ for $y \in Y$, so it is a nonempty convex subset of $\Pr(X\times Y)$. It is closed for the vague topology, because $\pi_\ast$ acts continuously on measures for this topology.  Also, for any $f \in L^\infty(\mu)$ and $h \in C(Y)$, one may find $g_n \in C(X)$ such that $g_n\to f$ in $L^1(\mu)$, and it follows that the functional on $Q(\mu,Y)$ defined by
\[\l \mapsto \int_{X\times Y} f(x)h(y)\,\l(\d x,\d y)\]
is the uniform limit of the functionals
\[\l \mapsto \int_{X\times Y} f_n(x)h(y)\,\l(\d x,\d y).\]
Therefore, the former functional is continuous on $Q(\mu,Y)$, even if $f$ is not itself continuous.

This, this restricted topology on $Q(\mu,Y)$ is something of a hybrid between the vague topology and the joining topology (see, for instance,~\cite[Section 6.1]{Gla03}), so we refer to it as the \textbf{hybrid topology}.

\subsection{Finite-complexity systems, and construction of a new group}

Theorem~\ref{thm:A} will be proved by induction on the `complexity' of the tuple of functions $p_1$, \ldots, $p_k$.  A key feature of the induction is that the conversion to a simpler tuple of functions is achieved at the expense of greatly enlarging the group.  Our basic notation and definitions will follow~\cite[Section 3]{Walsh12}.

Fix now a countable discrete group $G$.  A \textbf{$G$-sequence} is a function $\bbZ\to G$, and a \textbf{$G$-system} is a finite tuple of $G$-sequences.  The set $G^\bbZ$ of $G$-sequences is itself a group with coordinate-wise operations, and it is naturally endowed with the \textbf{shift automorphism} $\a$:
\[\a(p)(n) := p(n-1).\]
Also, let $\iota:G\to G^\bbZ$ be the embedding as the group of constant functions, which is precisely the subgroup of elements of $G^\bbZ$ fixed by $\a$. We will write $e$ for the identity element of both $G$ and $G^\bbZ$.

If $p \in G^\bbZ$, then
\[D_mp(n) := p(n+m)^{-1}p(n) = \a^{-n}(p^{-1})(m)\cdot p(n),\]
and if also $q \in G^\bbZ$, then
\[\langle p|q\rangle_m(n) := q(n+m) \cdot D_mp(n) = \a^{-n}(q p^{-1})(m) \cdot p(n).\]
Importantly, in this work we will consider these brackets combined into a single map
\[\langle p|q\rangle:\bbZ\to G^\bbZ:n\mapsto (\langle p|q\rangle_m(n))_{m\in\bbZ} = \a^{-n}(qp^{-1})\cdot \iota(p(n)),\] 
whereas Walsh just works with all the maps $n\to \langle p|q\rangle_m$ for $m\in\bbZ$ separately.  Note also that the multiplication in our expressions here is reversed compared to Walsh's: this is because we focus on the group action on the underlying space $X$, whereas he works consistently with the associated Koopman representation.

In terms of this construction, if $\bf{p} := (p_1,\ldots,p_k)$ is a $G$-system, then its \textbf{reduction} is the $G^\bbZ$-system
\[\bf{p}^\ast := (\iota\circ p_1,\ldots, \iota\circ p_{k-1},\langle p_k|e\rangle,\ldots,\langle p_k|p_{k-1}\rangle).\]

A $G$-system $\bf{p}$ is trivial if all its entries are constant (that is, elements of $\iota(G)$).  The $G$-systems has \textbf{finite complexity} if it may be converted into a trivial system by a finite sequence of operations, each of which is either a reduction, a re-ordering, or a removal of duplicated entries. The value of its \textbf{complexity} is the least number of reductions needed in this sequence of operations to reach a tuple of constants.  This is essentially~\cite[Definition 3.1]{Walsh12}, except that at each step Walsh must consider all the possible reductions
\[\bf{p}_m^\ast := (p_1,\ldots, p_{k-1},\langle p_k|e\rangle_m,\ldots,\langle p_k|p_{k-1}\rangle_m)\]
for different $m \in \bbZ$, whereas we combine these into a single $G^\bbZ$-system.  It is easy to see that our notion of complexity still agrees with his.

It is shown in~\cite[Theorem 4.2]{Walsh12} that every tuple of polynomial mappings into a nilpotent group has finite complexity.  Therefore, as in that paper, Theorem~\ref{thm:A} is a special case of the following:

\begin{thm}\label{thm:B}
Suppose that $G$ is a countable discrete group, that $T:G\actson (X,\mu)$ is a probability-preserving, topological action on a compact metric space, that $\bf{p} = (p_1,\ldots,p_k):\bbZ \to G$ is a finite-complexity $G$-system, and that $f_1,\ldots,f_k \in L^\infty(\mu)$.  Then the averages
\[\frac{1}{N}\sum_{n=1}^N (f_1\circ T^{p_1(n)})(f_2\circ T^{p_2(n)})\cdots (f_k\circ T^{p_k(n)})\]
converge in $\|\cdot\|_2$ as $N\to\infty$.
\end{thm}

\section{Completing the induction}\label{sec:gen}

The proof of Theorem~\ref{thm:B} will be by induction on the complexity of the $G$-system $(p_1,\ldots,p_k)$.  For the associated sequence of averages as in Theorem~\ref{thm:B}, we will show that their convergence in $L^2(\mu)$ is implied by the convergence of some analogous averages for the reduction $\bf{p}^\ast$.  These latter averages will correspond to an action, not of $G$, but of the group $\t{G} \leq G^\bbZ$ defined as the smallest subgroup of $G^\bbZ$ which contains $\iota(G)$, contains $p_i$ for each $i$, and is globally $\a$-invariant. This $\t{G}$ is a countable group, and $\bf{p}^\ast$ defines a $\t{G}$-system, because every entry of $\bf{p}^\ast$ takes values in $\t{G}$.

The base case of the induction is that in which each $\bf{p}$ has complexity zero, which implies that each $p_i$ is a constant element of $G$.  In this case the averages of interest do not depend on $N$, so convergence is trivial.  Thus, we now fix $G$ and a $G$-system $(p_1,p_2,\ldots,p_k)$ of complexity at least $1$, and assume as our inductive hypothesis that the conclusion of Theorem~\ref{thm:B} is already known for any group action and any system whose complexity is less than that of $\bf{p}$.  By definition of complexity, after possibly re-ordering $\bf{p}$ and removing duplicated entries (neither of which can disrupt the conclusion of Theorem~\ref{thm:B}), we may assume that $\bf{p}^\ast$ has complexity strictly less than $\bf{p}$, so this inductive hypothesis appliess to it.

\subsection{Canonical processes and basic functions}

We will continue to write $\t{G}$ for the subgroup of $G^\bbZ$ defined above.  In addition, since $\a(\t{G}) = \t{G}$, we may define from $\t{G}$ the semi-direct product $\t{H}:=\t{G}\rtimes_\a \bbZ$, which we identify as $\t{G}\times \a^\bbZ$ with the product
\[(p,\a^n)\cdot (p',\a^{n'}):= (p\a^n(p'),\a^{n+n'}).\]
This may be identified with the group of permutation of $\t{G}$ generated by $\a$ and by the left-regular representation of $\t{G}$ on itself.  We write $\rho$ for this permutation representation of $\t{H} = \t{G}\times \a^\bbZ$, so
\[\rho(q,\a^n):p\mapsto q\a^n(p).\]

Now let $K$ be some auxiliary compact metric space. Adopting a term from probability theory, a \textbf{$K$-valued canonical process} will a probability measure $\nu$ on $K^{\t{G}}$ which is invariant under the coordinate-permuting action $S:\t{H}\actson K^{\t{G}}$ arising from $\rho$.  This is generated by the transformations
\begin{eqnarray}\label{eq:Sdef1}
S^r((y_q)_{q\in \t{G}}) := (y_{r^{-1}q})_{q\in \t{G}}
\end{eqnarray}
and
\[S^\a((y_q)_{q \in \t{G}}) := (y_{\a^{-1}(q)})_{q \in \t{G}}.\]\
We will write $S^\iota$ for the $G$-subaction on $K^{\t{G}}$ defined by $(S^\iota)^g := S^{\iota(g)}.$

Let $\phi_q:K^{\t{G}}\to K$ be the projection onto the coordinate indexed by $q \in \t{G}$.  In the sequel we will need the case $K = [-1,1]^k$, for which we correspondingly write $\phi_q = (\phi_q^1,\ldots,\phi_q^k)$.

For our purposes, the first important feature of canonical processes $(\phi_q)_q$ is that the orbits of the individual functions $\phi_q$ under the action of $\t{H}$ satisfy certain algebraic relations.

\begin{lem}\label{lem:canon-rearrange}
Suppose that $p,r \in \t{G}$ and $n \in \bbZ$.  Then
\[\phi_{rp^{-1}}\circ S^{\iota(r(n))}\circ S^{\a^n} = \phi_e\circ S^{\langle r|p\rangle(n)}.\]
\end{lem}

\begin{proof}
If $y = (y_q)_q \in K^{\t{G}}$, then
\begin{eqnarray*}
\phi_{rp^{-1}}\big(S^{\iota(r(n))}\big(S^{\a^n}((y_q)_q)\big)\big) &=& \phi_{rp^{-1}}\big(S^{\iota(r(n))}\big((y_{\a^{-n}(q)})_q\big)\big)\\
&=& \phi_{rp^{-1}}((y_{\iota(r(n))^{-1}\a^{-n}(q)})_q)\\
&=& y_{\iota(r(n))^{-1}\a^{-n}(rp^{-1})}\\
&=& y_{(\langle r|p\rangle (n))^{-1}e} = \phi_e\circ S^{\langle r|p\rangle(n)}((y_q)_q).
\end{eqnarray*}
\end{proof}

Canonical processes will appear in our main proof via the following notion.

\begin{dfn}\label{dfn:basic}
A \textbf{basic function} on $(X,\mu,T)$ is a function $g \in L^\infty(\mu)$ for which there exist
\begin{itemize}
 \item a canonical process $\nu \in \Pr^{\t{H}}(([-1,1]^k)^{\t{G}})$, and
\item a $(\mu,\nu)$-coupling $\l \in \Pr(X\times ([-1,1]^k)^{\t{G}})$ invariant under both $T\times S^\iota$ (the diagonal action of $G$) and $\rm{id}_X\times S^\a$,
\end{itemize}
such that
\begin{eqnarray}\label{eq:dfn-basic}
g = \sfE_\l\Big(\prod_{i=1}^{k-1}\phi^i_{p_kp_i^{-1}}\cdot \phi^k_{p_k}\,\Big|\,X\Big).
\end{eqnarray}
\end{dfn}

Basic functions constitute the analog in our setting of Walsh's `reducible functions'.  The connection with the notion of reducibility will appear in the proof of the following.

\begin{prop}\label{prop:basicconv}
Suppose that the convergence is known for any system of averages with complexity strictly less than that of $(p_1,\ldots,p_k)$, suppose that $f_1,\ldots,f_{k-1} \in L^\infty(\mu)$, and suppose that $g \in L^\infty(\mu)$ is a basic function.  Then the averages
\[\L_N(f_1,\ldots,f_{k-1},g) := \frac{1}{N}\sum_{n=1}^N (f_1\circ T^{p_1(n)}) \cdots (f_{k-1}\circ T^{p_{k-1}(n)})(g\circ T^{p_k(n)})\]
converge in $\|\cdot\|_2$ as $N \to \infty$.
\end{prop}

\begin{proof}
Let $g$ be as in~(\ref{eq:dfn-basic}) with canonical process $\nu$ and coupling $\l$, and let
\[\t{g} := \prod_{i=1}^{k-1}\phi^i_{p_kp_i^{-1}} \cdot \phi^k_{p_k},\]
so $g = \sfE_\l(\t{g}\,|\,X)$.  Also, let $Y := ([-1,1]^k)^{\t{G}}$; let
\[\t{T}:= T\times S^\iota:G\actson X\times Y \quad \hbox{and} \quad \t{S}^\a := \rm{id}_X\times S^\a \actson X\times Y;\]
let $\pi_X$ and $\pi_Y$ be the two coordinate-projections of $X\times Y$; and let $\t{f}_j := f_j\circ \pi_X$ for $j \leq k-1$. Then $\l$ defines an extension of $G$-spaces through the first coordinate projection:
\[\pi_X:(X\times Y,\l,\t{T}) \to (X,\mu,T).\]
Let $\t{\L}_N$ be the averages analogous to $\L_N$ defined on the extended system.  Then we have
\begin{multline}\label{eq:cond-exp}
\L_N(f_1,\ldots,f_{k-1},g) = \sfE_\l(\t{\L}_N(\t{f}_1,\ldots,\t{f}_{k-1},g\circ \pi_X)\,|\,X)\\ = \sfE_\l(\t{\L}_N(\t{f}_1,\ldots,\t{f}_{k-1},\t{g})\,|\,X).
\end{multline}

Since $\l$ is $\t{S}^\a$-invariant, one has
\[\sfE_\l(F\,|\,X) = \sfE_\l(F\circ \t{S}^{\a^n}\,|\,X) \quad \forall n \in \bbZ,\ F\in L^1(\l).\]
On the other hand, each of the functions $\t{f}_j$ is lifted through $\pi_X$, hence is invariant under $\t{S}^\a$.  Substituting the definition of $\t{\L}_N$ into the right-hand side of~(\ref{eq:cond-exp}) and applying these two facts, one obtains that that conditional expectation is equal to
\begin{eqnarray*}
&&\frac{1}{N}\sum_{n=1}^N \sfE_\l\big((\t{f}_1\circ \t{T}^{p_1(n)})\cdots (\t{f}_{k-1}\circ \t{T}^{p_{k-1}(n)})(\t{g}\circ \t{T}^{p_k(n)})\,\big|\,X\big)\\
&&= \frac{1}{N}\sum_{n=1}^N \sfE_\l\big(((\t{f}_1\circ \t{T}^{p_1(n)}) \cdots (\t{f}_{k-1}\circ \t{T}^{p_{k-1}(n)})(\t{g}\circ \t{T}^{p_k(n)}))\t{S}^{\a^n}\,\big|\,X\big)\\
&&= \frac{1}{N}\sum_{n=1}^N \sfE_\l((\t{f}_1\circ \t{T}^{p_1(n)}) \cdots (\t{f}_{k-1}\circ \t{T}^{p_{k-1}(n)})(\t{g}\circ \t{T}^{p_k(n)}\circ \t{S}^{\a^n})\,|\,X).
\end{eqnarray*}
Since $\sfE_\l(\,\cdot\,|\,X): L^2(\l)\to L^2(\mu)$ is a contraction, it therefore suffices to prove that the averages
\[\Delta_N := \frac{1}{N}\sum_{n=1}^N (\t{f}_1\circ \t{T}^{p_1(n)}) \cdots (\t{f}_{k-1}\circ \t{T}^{p_{k-1}(n)})(\t{g}\circ \t{T}^{p_k(n)}\circ \t{S}^{\a^n})\]
converge in $L^2(\l)$ as $N \to \infty$.

However, now observe that
\begin{eqnarray*}
(\t{g}\circ \t{T}^{p_k(n)}\circ \t{S}^{\a^n}) (x,y) &=& \prod_{i=1}^{k-1}\phi^i_{p_kp_i^{-1}}(S^{\iota(p_k(n))}S^{\a^n}y)\cdot \phi^k_{p_k}(S^{\iota(p_k(n))}S^{\a^n}y)\\
&=& \prod_{i=1}^{k-1}\phi^i_e(S^{\langle p_k|p_i\rangle(n)}y)\cdot \phi^k_e(S^{\langle p_k|e\rangle(n)}y),
\end{eqnarray*}
by Lemma~\ref{lem:canon-rearrange}.

Therefore
\[\Delta_N = \frac{1}{N}\sum_{n=1}^N\prod_{i=1}^{k-1}(\t{f}_i\circ \t{T}^{p_i(n)})\cdot\prod_{i=1}^{k-1} (\phi^i_e \circ S^{\langle p_k|p_i\rangle(n)}\circ \pi_Y)\cdot (\phi^k_e\circ  S^{\langle p_k|e\rangle(n)}\circ\pi_Y).\]
Finally, applying Theorem~\ref{thm:recoverG} gives an extension $(Z,\theta,R)\stackrel{\xi}{\to} (Y,\nu,S^{\uhr\t{G}})$ of $\t{G}$-spaces for which there is a commutative diagram of $G$-spaces
\begin{center}
$\phantom{i}$\xymatrix{
(Z,\theta,R^{\uhr \iota(G)}) \ar^{\xi}[rr]\ar_\a[dr] && (Y,\nu,S^\iota)\\
& (X\times Y,\l,T\times S^\iota)\ar_{\pi_Y}[ur].
}
\end{center}

Lifting the averages $\Delta_N$ through $\a$ to the $\t{G}$-space $(Z,\theta,R)$, they become a sequence of multiple averages corresponding to the $\t{G}$-system $\bf{p}^\ast$, which we assumed has strictly smaller complexity than $\bf{p}$, so this convergence follows from the inductive hypothesis. (This last appeal to Theorem~\ref{thm:recoverG} is needed because, in the expression for $\Delta_N$ before making this extension, $\t{f}_i$ is not lifted from a function on $Y$ alone, but on the other hand $\langle p_k|p_i\rangle(n)$ acts only on the $Y$-coordinate.)
\end{proof}

\subsection{Completion of the proof}

\begin{prop}\label{prop:corn}
Suppose that $f_j \in L^\infty(\mu)$ for all $j=1,\ldots,k-1$, that $f_k \in L^2(\mu)$, and that
\[\|\L_N(f_1,\ldots,f_k)\|_2 \notto 0 \quad \hbox{as}\ N\to\infty.\]
Then there is a basic function $g \in L^\infty(\mu)$ such that $\langle f_k,g\rangle \neq 0$.
\end{prop}

Effectively, this proposition shows that the orthogonal projection onto the subspace of $L^2(\mu)$ generated by all the basic functions is `partially characteristic', in the terminology of~\cite[Section 3]{FurWei96}.

\begin{proof}
The desired correlation will follow if we find instead a canonical process $\nu$ and a $(\mu,\nu)$-coupling $\l$ as in Definition~\ref{dfn:basic} such that
\[\int_{X\times Y} f_k(x)\t{g}(y)\,\l(\d x,\d y) \neq 0,\]
where
\begin{eqnarray}\label{eq:gtilde}
\t{g} := \prod_{j=1}^{k-1}\phi^i_{p_kp_j^{-1}}\cdot \phi^k_{p_k}.
\end{eqnarray}

\vspace{7pt}

\emph{Step 1.}\quad Multiplying by constants if necessary, we may assume that $\|f_j\|_\infty \leq 1$ for each $j \leq k-1$.  Having done so, for any $f' \in L^2(\mu)$ and $N\geq 1$, one has
\[\|\L_N(f_1,\ldots,f_{k-1},f_k) - \L_N(f_1,\ldots,f_{k-1},f')\|_2 \leq \|f_k - f'\|_2.\]
Therefore, if we choose $f' \in L^\infty(\mu)$ so that $\|f_k - f'\|_2$ is sufficiently small, then our assumption of non-convergence to zero also implies
\[\limsup_{N\to\infty}\big\langle \L_N(f_1,\ldots,f_{k-1},f_k),\L_N(f_1,\ldots,f_{k-1},f')\big\rangle > 0.\]
Now letting $f''$ be a sufficiently small scalar multiple of $f'$, we may assume this nonzero limit supremum with $f'$ replaced by $f''$, and with $\|f''\|_\infty \leq 1$.

Let $A_N := \L_N(f_1,\ldots,f_{k-1},f'')$ for each $N$.

\vspace{7pt}

\emph{Step 2.}\quad Writing out the above inner products more completely, we obtain a subsequence $N_1 < N_2 < \ldots$ and some $\delta > 0$ such that
\begin{eqnarray*}
&& \langle \L_{N_i}(f_1,\ldots,f_k),A_{N_i}\rangle\\ &&= \frac{1}{N_i}\sum_{n=1}^{N_i} \int(f_1\circ T^{p_1(n)})\cdots (f_k\circ T^{p_k(n)})\cdot A_{N_i}\,\d\mu\\
&&= \int f_k\cdot \Big(\frac{1}{N_i}\sum_{n=1}^{N_i} \Big(\prod_{j=1}^{k-1}(f_j\circ T^{p_j(n)p_k(n)^{-1}})\cdot (A_{N_i}\circ T^{p_k(n)^{-1}})\Big)\Big)\,\d\mu\\
&&\to \delta.
\end{eqnarray*}

We will turn this into the desired correlation with $\t{g}$ by interpreting these last integrals as correlations with respect to a sequence of approximate couplings, from which we will then obtain $\l$ as a subsequential limit.

Thus, for each $N$, consider the measure on $X\times Y$ defined by
\[\l_{N_i} := \int_X \frac{1}{N}\sum_{n=1}^N \delta_{\big(x,\,(f_1(T^{q(n)^{-1}}x),\ldots,f_{k-1}(T^{q(n)^{-1}}x),A_{N_i}(T^{q(n)^{-1}}x))_{q\in \t{G}}\big)}\,\mu(\d x).\]
In terms of these, a simple re-arrangement gives
\[\langle \L_{N_i}(f_1,f_2,\ldots,f_k),A_{N_i}\rangle = \int_{X\times Y}f_k\cdot \t{g}\,\d\l_{N_i},\]
where $\t{g}$ is as in~(\ref{eq:gtilde}).

Clearly each $\l_{N_i}$ has marginal $\mu$ on $X$, so this is a sequence in $Q(\mu,Y)$.  By replacing $(N_i)_i$ with a subsequence if necessary, we may therefore assume that $\L_{N_i}\to \l \in Q(\mu,Y)$ in the hybrid topology of Section~\ref{sec:background}.  By the definition of that topology, our assumption of non-convergence to zero now implies
\[\int_{X\times Y}f_k\cdot \t{g}\,\d\l_{N_i} \to \int_{X\times Y}f_k\cdot \t{g}\,\d\l = \delta \neq 0.\]
Thus, letting $\nu$ be the marginal of $\l$ on $Y$, it remains to prove the following:
\begin{itemize}
\item[i)] $\l$ is $(T^g\times S^{\iota(g)})$-invariant for all $g\in G$;
\item[ii)] $\l$ is $(\rm{id}\times S^\a)$-invariant;
\item[iii)] $\nu$ is $S$-invariant.
\end{itemize}

\vspace{7pt}

\emph{Step 3.(i).}\quad For any $g \in G$, we have
\begin{eqnarray*}
&&(T^g\times S^{\iota(g)})_\ast \l_{N_i}\\
&&= \int_X \frac{1}{N_i}\sum_{n=1}^{N_i} \delta_{\big(T^gx,S^{\iota(g)}\big((f_1(T^{q(n)^{-1}}x),\ldots,f_{k-1}(T^{q(n)^{-1}}x),A_{N_i}(T^{q(n)^{-1}}x))_q\big)\big)}\,\mu(\d x)\\
&&= \int_X \frac{1}{N_i}\sum_{n=1}^{N_i} \delta_{\big(T^gx,(f_1(T^{q(n)^{-1}}T^gx),\ldots,f_{k-1}(T^{q(n)^{-1}}T^gx),A_{N_i}(T^{q(n)^{-1}}T^gx))_q\big)}\,\mu(\d x)\\
&&= \l_{N_i},
\end{eqnarray*}
where the second equality results from~(\ref{eq:Sdef1}) and the third from the $T^g$-invariance of $\mu$.  This invariance now persists under taking the hybrid limit.

\vspace{7pt}

\emph{Step 3.(ii).}\quad On the other hand,
\begin{eqnarray*}
&&(\rm{id}\times S^\a)_\ast \l_{N_i}\\
&&= \int_X \frac{1}{N_i}\sum_{n=1}^{N_i} \delta_{\big(x,(f_1(T^{(\a^{-1} q)(n)^{-1}}x),\ldots,f_{k-1}(T^{(\a^{-1} q)(n)^{-1}}x),A_{N_i}(T^{(\a^{-1} q)(n)^{-1}}x))_q\big)}\,\mu(\d x)\\
&&= \int_X \frac{1}{N_i}\sum_{n=1}^{N_i} \delta_{\big(x,(f_1(T^{q(n+1)^{-1}}x),\ldots,f_{k-1}(T^{q(n+1)^{-1}}x),A_{N_i}(T^{q(n+1)^{-1}}x))_q\big)}\,\mu(\d x),
\end{eqnarray*}
this time using the definition of $\a$.  Therefore
\[\|(\rm{id}\times S^\a)_\ast \l_{N_i} - \l_{N_i}\|_{\rm{TV}} = \rm{O}(1/N_i),\]
using the F\o lner property of the discrete intervals $[N_i]$, so in the limit the measure $\l$ is $(\rm{id} \times S^\a)$-invariant.

\vspace{7pt}

\emph{Step 3.(iii).}\quad Letting $\nu_{N_i}$ be the marginal of $\l_{N_i}$ on $Y$, the definition of $\l_{N_i}$ gives
\[\nu_{N_i} := \int_X \frac{1}{N_i}\sum_{n=1}^{N_i} \delta_{(f_1(T^{q(n)^{-1}}x),\ldots,f_{k-1}(T^{q(n)^{-1}}x),A_{N_i}(T^{q(n)^{-1}}x))_q}\,\mu(\d x).\]
Therefore, if $r \in \t{G}$, then
\[S^r_\ast\nu_{N_i} = \frac{1}{N_i}\sum_{n=1}^{N_i} \int_X \delta_{(f_1(T^{q(n)^{-1}r(n)}x),\ldots,f_{k-1}(T^{q(n)^{-1}r(n)}x),A_{N_i}(T^{q(n)^{-1}r(n)}x))_q}\,\mu(\d x),\]
and each integral in the right-hand average equals $\nu_{N_i}$ because $T^{r(n)}$ preserves $\mu$ for each $n$.  Therefore $\nu = \lim_i\nu_{N_i}$ is $S^{\t{G}}$-invariant.

On the other hand, the $S^\a$-invariance of $\nu$ follows from the fact that $\l$ is $(\rm{id}\times S^\a)$-invariant.
\end{proof}

\begin{proof}[Proof of Theorem~\ref{thm:A}]
Following the remarks at the beginning of this section, we need only show how the induction closes on itself for $(p_1,\ldots,p_k)$.  Let $V$ be the closed subspace of $L^2(\mu)$ generated by all the basic functions, and let $P:L^2(\mu)\to V$ be the orthogonal projection.  By multilinearity,
\[\L_N(f_1,f_2,\ldots,f_k) = \L_N(f_1,f_2,\ldots,Pf_k) + \L_N(f_1,f_2,\ldots,f_k - Pf_k).\]

The function $f_k - Pf_k$ is orthogonal to all basic functions, to the second term on the right must tend to zero in $L^2(\mu)$, by Proposition~\ref{prop:corn}.  On the other hand, $Pf_k$ may be approximated in $\|\cdot\|_2$ by finite linear combinations of basic functions.  Therefore, using multilinearity and a simple approximation, it suffices to prove convergence when $f_k$ is itself a basic function. This was the content of Proposition~\ref{prop:basicconv}.
\end{proof}

\begin{rmks}
\emph{1.}\quad As remarked previously, Proposition~\ref{prop:corn} is effectively proving that the closed subspace of $L^2(\mu)$ generated by the basic functions is partially characteristic for the averages $\L_N$.  It is worth contrasting this with previous uses of this idea, starting implicitly with Furstenberg's original work~\cite{Fur77} on Szemer\'edi's Theorem, and explicitly with~\cite{FurWei96}.  As far as I know, in all of those earlier works, the partially characteristic closed subspaces of $L^2(\mu)$ that appear are actually the subspaces of functions measurable with respect to a partially characteristic \emph{$\s$-subalgebra} (usually, but not always, a factor: see~\cite{Aus--commute}) of $(X,\mu)$.  However, this may not be the case in our setting.  This is because, given two basic functions, they may be defined in terms of two different couplings, and so it is not clear that their product is still a basic function.  I do not see any easy way to combine those two defining couplings into a single coupling that gives the product.  Therefore the space of bounded basic functions may not form an algebra, as it would if this subspace where defined by measurability with respect to some $\s$-subalgebra.  On the other hand, I also do not see how to generalize the proof of Proposition~\ref{prop:basicconv} to the case in which $g$ is a product of more than one basic function.  Thus, it seems to be important that we work with precisely the partially characteristic subspace spanned by the basic functions, and not, say, the $\s$-algebra that it generated.

\vspace{7pt}

\emph{2.}\quad The class of basic functions is quite mysterious.  The key to our re-incarnation of Walsh's proof is that their soft definition in terms of couplings with canonical processes is enough to simplify the averages of interest.  However, earlier ergodic-theoretic works on non-conventional averages have sought to give also a description of the possible limits of those averages, and of the factors of the original system that are responsible for them.  This can then be useful, for example, for proving new multiple recurrence phenomena.  Most famously, the results of Host and Kra in~\cite{HosKra05} give a fairly complete description for powers of a fixed transformation in terms of rotations on nilmanifolds.  I suspect that among the objects that appear in the proof above, the key to obtaining more structural information about basic functions is in describing canonical processes themselves.  Most crucially, it is not clear what constraints are imposed on the structure of a $\t{G}$-indexed, stationary stochastic process by assuming that its law is also invariant under $S^\a$. \fin
\end{rmks}

\bibliographystyle{abbrv}
\bibliography{bibfile}

\end{document}